\documentclass[american]{amsart}

\usepackage{amssymb} 
\usepackage{amsmath} 
\usepackage{amsfonts}
\usepackage[T1]{fontenc}
\usepackage[utf8]{inputenc}
\usepackage[english]{babel}
\usepackage{latexsym}
\usepackage{psfrag}
\usepackage{graphicx}

\usepackage{xcolor}
\usepackage[all]{xy}

\theoremstyle{definition} 

 \newtheorem{definition}{Definition}[section]
 \newtheorem{remarque}[definition]{Remark}
  
  \newtheorem{prob}[definition]{Problem}

\theoremstyle{plain}      

 \newtheorem{proposition}[definition]{Proposition}
 
 \newtheorem{corollaire}[definition]{Corollary}

\everymath{\displaystyle}

\DeclareMathOperator{\ii}{i}
\DeclareMathOperator{\id}{id}

\DeclareMathOperator{\im}{Im}
\DeclareMathOperator{\Co}{\mathbb{C}}
\DeclareMathOperator{\dz}{\partial_{z}}
\DeclareMathOperator{\dx}{\partial_{x}}
\DeclareMathOperator{\dy}{\partial_{y}}
\DeclareMathOperator{\dzb}{\partial_{\overline{z}}}
\DeclareMathOperator{\Mod}{mod}
\DeclareMathOperator{\jac}{Jac}
\DeclareMathOperator{\Kf}{\textit{K}}
\DeclareMathOperator{\bd}{\partial\!\mathbb{D}}
\DeclareMathOperator{\dd}{\mathbb{D}}

\usepackage{indentfirst}

\usepackage{color}
\definecolor{darkred}{rgb}{0.9,0.,.2}
\definecolor{darkblue}{rgb}{0.,0.,.6}
\definecolor{darkgreen}{rgb}{0.,.6,0.1}

\newcommand{\eq}[1][r]
   {\ar@<-3pt>@{-}[#1]
    \ar@<-1pt>@{}[#1]|<{}="gauche"
    \ar@<+0pt>@{}[#1]|-{}="milieu"
    \ar@<+1pt>@{}[#1]|>{}="droite"
    \ar@/^2pt/@{-}"gauche";"milieu"
    \ar@/_2pt/@{-}"milieu";"droite"}

\def\commutatif{\ar@{}[rd]|{\circlearrowleft}}

\begin{document}

\title[Commentary]{A Commentary on Teichm\"uller's paper \emph{Verschiebungssatz der quasikonformen Abbildung} Deutsche Math. 7 (1944), 336-343}

\author{Vincent Alberge}
\address{
Institut de Recherche Math\'ematique Avanc\'ee\\ CNRS et  Universit\'e de Strasbourg\\\small 7 rue Ren\'e
  Descartes - 67084 Strasbourg Cedex, France\\
email:\,\tt{alberge@math.unistra.fr}
}

\begin{abstract}
This is a  commentary on Teichm\"uller's paper \textit{Ein Verschiebungs\-satz der quasikonformen Abbildung} (A displacement theorem of quasiconformal mapping), published in 1944. We explain in detail how Teichm\"uller solves the problem of finding the quasiconformal mapping from the unit disc to itself, sending $0$ to  a strictly negative point on the real line, holding the boundary of the disc pointwise fixed and with the smallest quasiconformal dilatation. We mention also some consequences of this extremal problem and we ask a question.
\end{abstract}

\maketitle

\section{Introduction}

This is a commentary on Teichm\"uller's paper \textit{Ein Verschiebungssatz der quasi\-konformen Abbildung}, published in 1944\footnote{Note that this paper appeared after his death which occured in 1943 on the Eastern Front.} (see \cite{T31}). We refer to the English translation  which appears in this volume. The present paper is part of a series of commentaries written by various authors on papers of Teichm\"uller. These papers contain some ideas which are still unknown to Teichm\"uller theorists, see for example \cite{papadop&acampo&ji}, \cite{alberge&papadop&su} and \cite{T29C}.
The paper \cite{T31} is one of the last that Teichm\"uller wrote, and especially the last one about quasiconformal maps. In this paper, he solved the following geometric problem:

\begin{prob}\label{mainpb} Find and describe the quasiconformal map from the unit disc to itself such that 
\begin{itemize}
\item its restriction to the unit circle is the identity map,
\item the image of $0$ is $-x$, where $0<x<1$,
\item its quasiconformal dilatation is as small as possible.
\end{itemize}
\end{prob}

As Teichm\"uller wrote at the beginning of his paper, this extremal problem is rather different from those studied in \cite{T20}. It is due to the fact that mappings fix all boundary points and not only a finite number of such points. The paper \cite{T20} is at the foundation of the theory that we call now the \textit{classical Teichm\"uller theory}.

In order to solve Problem \ref{mainpb}, Teichm\"uller used an  idea already contained in §$23$ and §$24$ of \cite{T20}.\footnote{Even if the idea was already used, the result was not known from specialists. We refer to \cite{kuhnau} and especially to what Grötzsch told to K\"uhnau about this paper: ``Ja...ah, das habe ich nicht gehabt [Okay, this I did not have].''} Indeed, he obtained an equivalent problem (see Problem \ref{pbb} in Subsection \ref{subsection}) by taking ramified coverings, which turns out to be  simpler. 
Let us say a few words about that. First, he constructed, using explicit conformal maps, two $2$-sheeted branched coverings of the unit disc,  the first one branched at $0$ and the other at $-x$. These two covering spaces can be conformally represented by ellipses with  data depending on $x$. This construction shows that the main problem is equivalent to a problem of minimization of the quasiconformal dilatation for mappings between  two ellipses with a particular condition on the boundary. After that, he showed, using the  Cauchy-Schwarz inequality (as in the solution to the Gr\"otzsch problem), that the extremal map between the ellipses is given by an affine transformation. Finally, he gave for the quasiconformal dilatation of the extremal map a lower bound depending on $x$ and an asymptotic behaviour when $x$ approaches $0$. 

In \cite{T31}, Teichm\"uller did not give definitions; all the definitions he used are in \cite{T20}. This is why in these notes we will change the organization of the text in comparison with \cite{T31}; but we will keep all the ideas from Teichm\"uller.  

After recalling some notation and definitions, especially about quasiconformal mappings, we will explain the Gr\"otzsch problem and we will recall the notion of Gr\"otzsch module. We will then give  details on the proof of Teichm\"uller. We will conclude by some applications of this result.

\bigskip

\noindent {\bf Acknowledgements.} The author would like to express his sincere gratitude to Professor Athanase Papadopoulos for giving an opportunity for writing this commentary and especially for his patience and kindness. The author wants to thanks Professor Ken'ichi Ohshika for his helpful comments on a previous version. The author also thanks Professor Irwin Kra for his interest and for a correspondence on these results. The author would like to thank Professor Hideki Miyachi for his suggestions. 

This work was partially supported by the French ANR grant FINSLER. The author acknowledges support from U.S. National Science Foundation grants DMS 1107452, 1107263, 1107367 "RNMS: GEometric structures And Representation varieties" (the GEAR Network).

\section{Preliminaries}\label{section}

All along this paper, we shall be interested in planar quasiconformal mappings. Unless otherwise noted, all domains that we consider are  connected subsets of the Riemann sphere $\overline{\mathbb{C}}:=\mathbb{C}\cup\left\lbrace \infty \right\rbrace$. There are several books which deal with quasiconformal mappings, see e.g. \cite{ahlforsqc}, \cite{lehto&virtanen} or \cite{fletcher&markovic}.

We give below two equivalent definitions of quasiconformal maps. Both of them are interesting and they introduce notions (module and quasiconformal dilatation) that Teichm\"uller used to solve Problem \ref{mainpb}. 

A \emph{quadrilateral} $\mathcal{Q}$ is a Jordan domain (i.e. a simply connected domain whose boundary is a Jordan curve) with four distinct boundary points. Sometimes, we will denote by $\mathcal{Q}\left( a, b , c ,d \right)$ such a quadrilateral, where $a$, $b$, $c$ and $d$ are boundary points and we shall usually assume that these four points appear on the boundary in that order. 
By applying successively the Riemann Mapping Theorem, the Carath\'{e}odory Theorem\footnote{The theorem referred to is known as the \textit{boundary correspondance theorem}.} and a suitable Schwarz-Christoffel mapping, we know that $\mathcal{Q}$ is conformally equivalent (i.e. there exists a holomorphic bijection) to a rectangle $\mathcal{R}$ of vertical side length $1$ and horizontal length side $m$, for some uniquely defined $m>0$.\footnote{To be more precise, the interior of $\mathcal{Q}$ is sent (conformally) onto the upper half-plane $\mathbb{H}$ and this map can be extended to a homeomorphism from $\mathcal{Q}$ to $\mathbb{H}\cup \mathbb{R}\cup \left\lbrace \infty \right\rbrace$. Moreover, the four distinguished points of $\mathcal{Q}$ are sent respectively to $0$, $1$, $\lambda$ and $\infty$ for some $\lambda>1$. Finally, $z\mapsto c\int_{z}{\frac{d\!\zeta}{\sqrt{ \zeta\left( \zeta - 1 \right) \left( \zeta - \lambda \right)}}}$ (for a suitable $c$) maps the upper half-plane onto the rectangle $\mathcal{R}$ of vertical length side $1$ and horizontal length side $m$.} We call the \textit{module} of $\mathcal{Q}$, denoted by $\Mod \left( Q \right)$, the number $m$. 

A \emph{doubly-connected domain} $\mathcal{C}$ is a connected domain whose boundary is the union of two disjoint Jordan curves. As for the quadrilateral, we can associate a module to a doubly-connected domain. We know that such a domain is conformally equivalent to an annulus whose inner radius is $1$ and outer radius is $R$, for some $R>1$. We call the \textit{module} of $\mathcal{C}$, denoted by $\Mod\left( \mathcal{C}\right)$, the number  $\frac{1}{2\pi}\log\left( R \right)$. 

\begin{remarque} Another way to introduce the module is to define it as the inverse of  \textit{extremal length} of a particular family of curves. This relation enables us to extend the notion of quasiconformal mapping in higher dimensions. 
\end{remarque}

\begin{definition}[Geometric definition\label{def1}]
Let $\Omega$ be an open set in $\Co$. Let $f : \Omega \rightarrow f\left( \Omega\right)$ be an orientation-preserving homeomorphism. We say that $f$ is \textit{quasiconformal} if there exists $\Kf \geq 1$ such that for any quadrilateral $\mathcal{Q}\subset\Omega$,

$$
\Mod\left( f\left( \mathcal{Q} \right)\right)\leq \Kf\cdot\Mod \left( \mathcal{Q} \right).
$$

In this case, we set $\Kf_f := \sup_{\mathcal{Q}}{\frac{\Mod\left( f\left( \mathcal{Q} \right)\right)}{\Mod \left( \mathcal{Q} \right)}}$ and we call it the \emph{quasiconformal dilatation}\footnote{This is not the term that Teichm\"uller used in his papers (for example \cite{T20}, \cite{T31} and \cite{T29}). He used the term ``dilatation quotient.'' In the current literature, we can also find the terms ``maximal dilatation,'' ``dilatation,'' ``distorsion'' or ``quasiconformal norm.''} of $f$. Moreover,  we say that $f$ is $\Kf_f$-quasiconformal.
\end{definition}

To simplify notation, we write \textit{q.c.} instead of quasiconformal.

With this definition, it is  easy to see that for $f_1$ and $f_2$ respectively $K_1$-q.c. and $K_2$-q.c. on suitable domains,  $f_1 \circ f_2$  is $K_1 K_2$-q.c.

We can show that $f$ is conformal if and only if $\Kf =1$. Thus, if $g$ and $h$ are conformal, then $g\circ f \circ h$ has the same q.c. dilatation as $f$. 

Before giving an equivalent definition of q.c. mappings, we recall that a map $f$ is \emph{absolutely continuous on lines} (\textit{ACL}) in a domain $\Omega$ if for every rectangle $R:=\left\lbrace x+\ii y \, \mid \, a<x<b,\, c<y<d \right\rbrace$ in $\Omega$, $f$ is absolutely continuous as a function of $x$ (resp. $y$) on almost all segments $I_y := \left\lbrace x+\ii y \,\mid\, a<x<b\right\rbrace$ (resp. $I_x := \left\lbrace x+\ii y \, \mid \, c<y<d \right\rbrace$).
 We can show that such a function $f$ is differientiable almost everywhere (a.e.) in $\Omega$.

The second equivalent definition of quasiconformality is the following.

\begin{definition}[Analytic definition]\label{def2}
Let $\Omega$ be an open set of $\Co$. Let $f : \Omega \rightarrow f\left( \Omega\right)$ be a homeomorphism. We say that $f$ is $\Kf$-quasiconformal if
\begin{enumerate}
\item $f$ is ACL on $\Omega$,
\item $\vert \dzb f \vert \leq k\cdot\vert \dz f \vert$  (a.e), where $k=\frac{K-1}{K+1}$.
\end{enumerate}
\end{definition}

We recall that 
\begin{align*}
\begin{cases}
\dz f &= \frac{1}{2}\left( \dx f -\ii\dy f \right), \\
\dzb f &= \dfrac{1}{2}\left( \dx f +\ii\dy f \right), \\ 
\jac \left( f \right) &= \vert \dz f \vert ^2 - \vert \dzb f \vert^2 .
\end{cases}
\end{align*}

For a q.c. mapping $f$ (in the sense of Definition \ref{def2}), we can show that 
$$
\Kf_{f}=\underset{z\in\Omega}{\textrm{ess.sup}}{\frac{\vert \dz f \left( z \right) \vert+\vert \dzb f \left( z \right) \vert}{\vert \dzb f \left( z \right)\vert-\vert \dzb f\left( z \right) \vert}}.
$$

According to the introduction of \cite{lehto&virtanen}, Definition \ref{def2} was introduced by Morrey in \cite{morrey}. Moreover, Definition \ref{def1} is due  to Ahlfors (see \cite{ahlforsqcintro}). Works by Bers, Mori and Yûjôbô\footnote{This list is not exhaustive.} show that Definition \ref{def1} and Definition \ref{def2} are equivalent. A proof can be found in \cite{ahlforsqc}. 

\sloppy
Gr\"otzsch introduced q.c. mappings in \cite{rd3} under the name ``nichtkonforme Abbildungen.''\footnote{The name \textit{quasiconformal} is now credited to Ahlfors, but this is debatable (see the commentary related to \cite{ahlforsoeuvrescompletes} p. 213).} His definition is similar to Definition \ref{def2} but with a stronger hypothesis: he assumes the maps  differentiable. Teichm\"uller was the first to use q.c. mappings in a substantial manner. He developed this theory to study  the \textit{Teichm\"uller space}\footnote{Teichm\"uller  called it ``the space of topologically determined principal regions.''} and  the so-called \textit{Teichm\"uller metric}. Note that he used in \cite{T20}, and \cite{T31} the Gr\"otzsch definition for q.c. mappings in a slightly different form. He considered mappings which are differentiable ``up to finitely many closed analytical curve segments.'' Since nothing changes in Teichm\"uller's proof (Proposition \ref{prop} below) we shall use Definition \ref{def2} for q.c. mappings. In particular, using the point of view of Definition \ref{def2}, we can show  that for any q.c. mapping $f$, $\dz f \neq 0$ (a.e.). 

\fussy

\section{Gr\"otzsch's problem and Gr\"otzsch's domain}

\subsection{Gr\"otzsch's problem}
To conclude the solution of Problem \ref{mainpb}, Teichm\"uller used the same idea as Gr\"otzsch used in \cite{rd4} to solve what we call now the Gr\"otzsch problem. In fact, Teichm\"uller used a kind of generalization of Gr\"otzsch's problem in his development of   classical Teichm\"uller theory. So, let us present this problem.

Let $\mathcal{R}_1$ and  $\mathcal{R}_2$ be two quadrilaterals. As already mentioned at the beginning of Section \ref{section}, we can suppose that $\mathcal{R}_1$ and $\mathcal{R}_2$ are rectangles whose modules are respectively $a_1$ and $a_2$ (see Figure \ref{fig1}).

\begin{center}
\begin{figure}[!ht]
\centering
\psfrag{a1}{$a_1$}
\psfrag{a2}{$a_2$}
\psfrag{R1}{$\mathcal{R}_1$}
\psfrag{R2}{$\mathcal{R}_2$}
\psfrag{1}{$1$}
\includegraphics[width=0.8\linewidth]{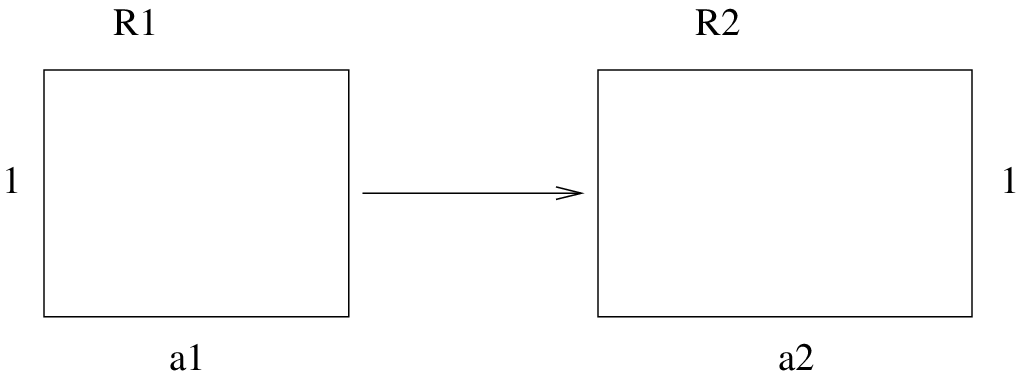}
\caption{}\label{fig1}
\end{figure}
\end{center}

The \emph{Gr\"otzsch problem} is the following: 
\begin{quote}
\textit{Find and describe the q.c. mapping from $\mathcal{R}_1$ to $\mathcal{R}_2$ which preserves sides and with the smallest q.c. dilatation.}
\end{quote}
First, we have to find a q.c. mapping from $\mathcal{R}_1$ to $\mathcal{R}_2$. The simplest map is the following   
$$z \mapsto \frac{1}{2}\left( 1+\frac{a_2}{a_1}\right)\cdot z + \frac{1}{2} \left( \frac{a_2}{a_1}-1\right)\cdot \overline{z},$$ and its q.c. dilatation is equal to $\max\left( \frac{a_1}{a_2} , \frac{a_2}{a_1} \right)$. By the Cauchy-Schwarz inequality, we check that if $f$ is a q.c. mapping between these two rectangles, then $\Kf_f \geq \max\left( \frac{a_1}{a_2} , \frac{a_2}{a_1} \right)$, with equality if and only if $f$ is the  affine map above. This solves this problem.

Teichm\"uller used the same principle. He found a q.c. mapping which can be a candidate and he showed in the same manner that this map is the one with the smallest q.c. dilatation.

\subsection{Gr\"otzsch's domain and its associated module}

A Gr\"otzsch domain is an extremal domain for the following problem. Let $R>1$ be a real number and let $\Omega\subset\mathbb{C}$ be a doubly-connected domain separating the unit circle $\partial\!\dd$ from $\left\lbrace R, \infty \right\rbrace$. Such a domain has a module and we want to know whether there exists  a domain whose module is maximal. The answer is yes, and we can describe it. 

This domain, now called a \emph{Gr\"otzsch domain}, is $\mathbb{C}\setminus\left( \overline{\dd}\cup\left[R, \infty\right)\right)$. We denote its module by $\frac{1}{2\pi}\log{\left(\Phi \left( R \right)\right)}$. Before Teichm\"uller, some facts about the map $\Phi$ were already known to Gr\"otzsch (see \cite{rd1, rd2}), like the fact that $\Phi : \left] 1, + \infty \right[ \rightarrow \left] 1, +\infty \right[$ is an increasing continuous function such that
\begin{equation}\label{encadr}
\forall R >1; \; \; R<\Phi\left( R \right) < 4R,
\end{equation}
and
\begin{equation}\label{asympt}
\lim_{R\rightarrow \infty}{\left( \Phi\left( R \right) - 4R \right)}=0.
\end{equation}

According to \cite{T31}, Teichm\"uller ``proved in a \textit{purely geometric way}'' properties (\ref{encadr}) and (\ref{asympt}). See \cite{teichmuller3} for these proofs.

We give below a  functional relation with a sketch of proof. Let $\alpha$ be a positive number strictly less than $1$. The doubly-connected domain $\mathbb{D}\setminus\left( \left[ -\alpha, \alpha\right]\right)$ has a module which is (with our notation) $$\frac{1}{2\pi}\log\left(\Phi\left( \frac{1}{2}\left( \alpha+\frac{1}{\alpha}\right) \right)\right).$$ Indeed, the image of $\mathbb{D}\setminus\left( \left[ -\alpha, \alpha\right]\right)$ by  the biholomorphism of $\mathbb{D}$ sending $-\alpha$ to $0$ and $ \alpha$ to $\frac{2\alpha}{1+\alpha^2}$ is  $\mathbb{D}\setminus\left[0, \frac{2\alpha}{1+\alpha^2} \right]$. By applying $z\mapsto 1/z$, we see that its module is exactly what we wrote. 
Moreover, the Gr\"otzsch domain associated with $1/\alpha^2$ is equivalent to $$\mathbb{C}\setminus\left(\left[ -1,1 \right]\cup\left[\frac{1}{2}\left( \alpha^2 +\frac{1}{\alpha^2}\right) , \infty \right) \right).$$ To see this, we use the map $z \mapsto \frac{1}{2}\left( z+\frac{1}{z}\right)$. Note that this map will be important in the solution of our problem and also that it is the bridge between the Gr\"otzsch domain and what is now called  the \textit{Teichm\"uller domain} (see Chapter 3 of \cite{ahlforsqc}). Now by $z\mapsto \frac{\alpha}{1+\alpha^2}\left( z+1\right)$,  we reach $$\mathbb{C}\setminus\left(\left[ 0,\left( \frac{1}{2}\left( \alpha+\frac{1}{\alpha}\right)\right)^{-1}\right]\cup \left[ \frac{1}{2}\left( \alpha+\frac{1}{\alpha}\right), \infty \right)\right).$$ This  domain has, by  reflection with respect to the unit circle, a module equal to $$\frac{1}{\pi}\log\left( \Phi\left( \frac{1}{2}\left( \alpha+\frac{1}{\alpha}\right)\right)\right).$$ Thus, we obtain the following relation
\begin{equation}\label{relationfonctionnelle}
\Phi\left( \frac{1}{2}\left( \alpha +\frac{1}{\alpha}\right)\right)=\sqrt{\Phi\left( \frac{1}{\alpha^2} \right)}.
\end{equation}

The expressions  ``Gr\"otzsch's domain'' and  ``Teichm\"u\-ller's domain'' are used by Ahlfors  in \cite{ahlforsqc}  and also by Lehto and Virtanen in \cite{lehto&virtanen} whose  German version was published in 1965. The author of this report does not know who was the first person to use this terminology.

\section{The solution of Teichm\"uller's problem}

\subsection{A simple case}\label{subsection}
For two strictly positive real numbers $\alpha$ and $\beta$, we denote by $\mathcal{E}\left( \alpha, \beta\right)$  the ellipse whose centre is the origin and the major (resp. minor) axis is equal to $\alpha$ (resp. $\beta$). We want to solve the following extremal problem:

\begin{prob}\label{pbb}
Is there a q.c. mapping from $\mathcal{E}\left( \alpha, \beta\right)$ to $\mathcal{E}\left( \beta, \alpha\right)$ with the smallest q.c. dilatation and whose  restriction to the boundary coincides with the restriction of $h_0 : z\mapsto \frac{1}{2}\left(\frac{\alpha}{\beta}+\frac{\beta}{\alpha} \right)\cdot z + \frac{1}{2}\left( \frac{\alpha}{\beta}-\frac{\beta}{\alpha}\right)\cdot \overline{z}$ ?
\end{prob}

It is easy to show that $h_0$ sends $\mathcal{E}\left( \alpha, \beta\right)$ to $\mathcal{E}\left( \beta, \alpha\right)$ with  the good behaviour at the boundary. Moreover, its q.c. dilatation is equal to $\max\left( \frac{\alpha^2}{\beta^2}, \frac{\beta^2}{\alpha^2} \right)$. We will show that this is the smallest q.c. dilatation with the given conditions.

\begin{proposition}\label{prop}
Let $f : \mathcal{E}\left( \alpha, \beta\right)\mapsto\mathcal{E}\left( \beta, \alpha\right) $ be 	a  q.c. mapping such that $f_{\mid_{\partial \mathcal{E}\left( \beta, \alpha\right)}}$ coincides with the restriction of $h_0$. Then,
$$
\Kf_{f}\geq \max\left( \frac{\alpha^2}{\beta^2}, \frac{\beta^2}{\alpha^2} \right).
$$
\end{proposition}

\begin{proof}
As in the solution of the Gr\"otzsch problem, we can suppose that $f$ is continuously differentiable in both directions. Let $y\in\left] -\beta, \beta\right[$. We denote by $l\left( y \right)$ the Euclidean length of the segment $\im \left( z \right)=y$ in $\mathcal{E}\left( \alpha, \beta\right)$. We parametrize this segment by $\gamma_y : t\in \left[ -\frac{l\left( y \right)}{2}, \frac{l\left( y \right)}{2} \right]$. Due to the hypothesis on $f_{\mid_{\partial \mathcal{E}\left( \beta, \alpha\right)}}$, the length of $f\circ\gamma_y$ is bigger than $\frac{\alpha}{\beta}\cdot l\left( y \right)$. We have the following inequality
\begin{equation}\label{eq2}
\frac{\alpha}{\beta}\cdot l\left( y \right)\leq \int_{-\frac{l\left( y \right)}{2}}^{\frac{l\left( y \right)}{2}}{\vert \left( f\circ \gamma_y \right)^\prime \left( t \right) \vert dt}.
\end{equation}
But 
$$
\left( f\circ \gamma_y \right)^\prime \left( t \right)= \dz\! f\left( \gamma_y \left( t \right)\right)\cdot \gamma_y ^\prime \left( t \right)+\dzb\!f \left( \gamma_y \left( t \right)\right)\cdot \overline{\gamma_y ^\prime \left( t \right)},
$$
so
\begin{align}\label{eq3}
\vert\left( f\circ \gamma_y \right)^{\!\prime} \left( t \right)\vert &\leq \left( \vert \dz\! f\left( \gamma_y \left( t \right)\right) \vert + \vert \dzb\! f\left( \gamma_y \left( t \right)\right)\vert\right)\cdot \vert \gamma_{y}^\prime \left( t \right)  \vert \nonumber\\
&=  \left( \frac{\vert \dz\!f \!\left( \gamma_y \left( t \right)\right) \vert+\vert \dzb\! f\! \left( \gamma_y \left( t \right) \right) \vert}{\vert \dz\!f\! \left( \gamma_y \left( t \right)\right) \vert-\vert \dzb\!f\! \left( \gamma_y \left( t \right) \right) \vert} \cdot \left( \vert \dz\!f\! \left( \gamma_y \left( t \right)\right) \vert^2 \right.\right. \nonumber\\
& \qquad \qquad\qquad\qquad\qquad\qquad\qquad \left. \left. -\vert \dzb\! f\! \left( \gamma_y \left( t \right) \right) \vert^2 \right)\!\right)^\frac{1}{2} \nonumber\\
&\leq \left( \Kf_{f}\cdot\jac\left( f \right) \left( \gamma_y \left( t \right) \right)\right)^\frac{1}{2}.
\end{align}
By applying the Cauchy-Schwarz inequality in (\ref{eq2}) and using (\ref{eq3}), we obtain
\begin{equation*}
\frac{\alpha^2}{\beta^2}\cdot l\left( y \right)\leq \Kf_f \int_{-\frac{l\left( y \right)}{2}}^{\frac{l\left( y \right)}{2}}{\jac\left( f \right)\left( \gamma_y \left( t \right)\right)dt}.
\end{equation*}
Integration with respect to $y$ gives us
$$
\frac{\alpha^2}{\beta^2}\leq \Kf_f .
$$
If we replace the horizontal segment by the vertical segment in the ellipse, by the same method we obtain 
$$
\frac{\beta^2}{\alpha^2}\leq \Kf_f ,
$$
and so, the proof is complete.
\end{proof}

\begin{center}
\begin{figure}[ht]
\centering
\psfrag{E}{$\mathcal{E}\left( \alpha, \beta\right)$}
\psfrag{F}{$\mathcal{E}\left( \beta, \alpha\right)$}
\psfrag{a}{$\alpha$}
\psfrag{b}{$\beta$}
\includegraphics[width=0.8\linewidth]{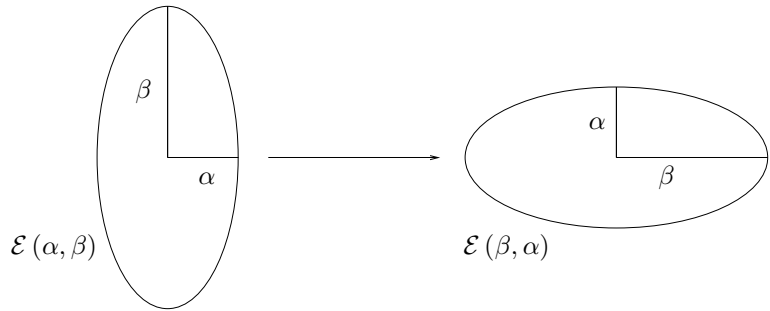}
\caption{Like in the case of rectangles, we are looking for the q.c. mapping with the smallest q.c. dilatation and a good behaviour at the boundary.}\label{fig2}
\end{figure}
\end{center}

\subsection{Solution}
We will explain how Teichm\"uller showed that Problem \ref{mainpb} is equivalent to Problem \ref{pbb} for some $\left( \alpha , \beta \right)$ that we shall specify. 

We start with $\overline{\mathbb{D}}\setminus \left[ -x, 0\right]$. Its preimage by the covering map, $z\mapsto z^2$ is  $\overline{\mathbb{D}}\setminus \ii\left[-\sqrt{x}, \sqrt{x}\right]$. The interior of the latter domain can be conformally sent by $\varphi$ onto an annulus of inner radius $1$ and outer radius $R$. To see this, we have to map (conformally) the first quadrant of the unit disc onto the first quadrant of this annulus for some $R$. Such a map exists by the same arguments given in Section 2. By successive reflections with respect to the horizontal and the  vertical axes, we can define $\varphi$. Note that according to (\ref{relationfonctionnelle}),
\begin{equation}\label{eqmod}
R=\sqrt{\Phi\left( \frac{1}{x}\right)}.
\end{equation}

We must say that Teichm\"uller used a nicer method to obtain (\ref{eqmod}). Indeed, this relation is given by the following commutative diagram 

\begin{equation*}
\xymatrix{ \mathbb{D}\setminus \ii\left[-\sqrt{x}, \sqrt{x} \right] \ar[rr]^{\varphi} \ar[dd]_{z\mapsto z^2} & & \mathcal{C}\left( 1, R \right) \ar@{.>}[dd] \\ && \\ \mathbb{D}\setminus \left[ -x, x\right] \ar[rr] && \mathcal{C}\left( 1, \Phi\left( \frac{1}{x} \right)\right)
}
\end{equation*}
where $\mathcal{C}\left( 1, R \right)$ (resp. $\mathcal{C}\left( 1, \Phi\left( \frac{1}{x} \right)\right)$) denotes the annulus whose  inner radius is $1$ and  outer radius is $R$ (resp. $\Phi\left( \frac{1}{x} \right)$).

Finally, $f_1 : z\mapsto z-\frac{1}{z}$ and $f_2 : z\mapsto z+\frac{1}{z}$ map the annulus $\mathcal{C}\left( 1, R \right)$ onto $\mathcal{E}\left( R-\frac{1}{R}, R+\frac{1}{R} \right)\setminus\ii\left[-2, 2\right]$ and  $\mathcal{E}\left( R+\frac{1}{R}, R-\frac{1}{R} \right)\setminus\left[-2, 2\right]$ respectively. To simplify notation, we set $\mathcal{E}_1 :=\mathcal{E}\left( R-\frac{1}{R}, R+\frac{1}{R} \right)$ and $\mathcal{E}_2 :=\mathcal{E}\left( R+\frac{1}{R}, R-\frac{1}{R} \right)$. Thus, we have two new maps, $p_1 := \left( \varphi^{-1}\circ f_1^{-1} \right)^2$ and $p_2 := \left( \varphi^{-1}\circ f_2^{-1} \right)^2$.  The mapping $p_1$ (resp. $p_2$) can be extended  to a map from $\mathcal{E}_1$ (resp. $\mathcal{E}_2$) to $\mathbb{D}$ such that $0$ is sent to $0$ (resp. $-x$). We denote the associated maps again by  $p_1$ and $p_2$. For more details, see Figure \ref{fig4}. Note that  in Teichm\"uller's paper \cite{T31}, there is an equivalent figure. 

Now, we remark that $p_1 : \mathcal{E}_1 \rightarrow \mathbb{D}$ (resp. $p_2 : \mathcal{E}_2 \rightarrow \mathbb{D}$) is a two-sheeted ramified covering, where the branch point is $0$ (resp. $-x$).

\begin{center}
\begin{figure}[!ht]
\centering
\psfrag{0}{$0$}
\psfrag{-x}{$-x$}
\psfrag{ix}{$\ii\sqrt{x}$}
\psfrag{-ix}{$-\ii\sqrt{x}$}
\psfrag{D}{$\mathbb{D}$}
\psfrag{g1}{$f_{1}^{-1}$}
\psfrag{g2}{$f_{2}^{-1}$}
\psfrag{fi}{$\varphi^{-1}$}
\psfrag{s}{$z\mapsto z^2$}
\psfrag{E1}{$\mathcal{E}_1$}
\psfrag{E2}{$\mathcal{E}_2$}
\psfrag{1}{$1$}
\psfrag{R}{$R$}
\includegraphics[width=0.82\linewidth]{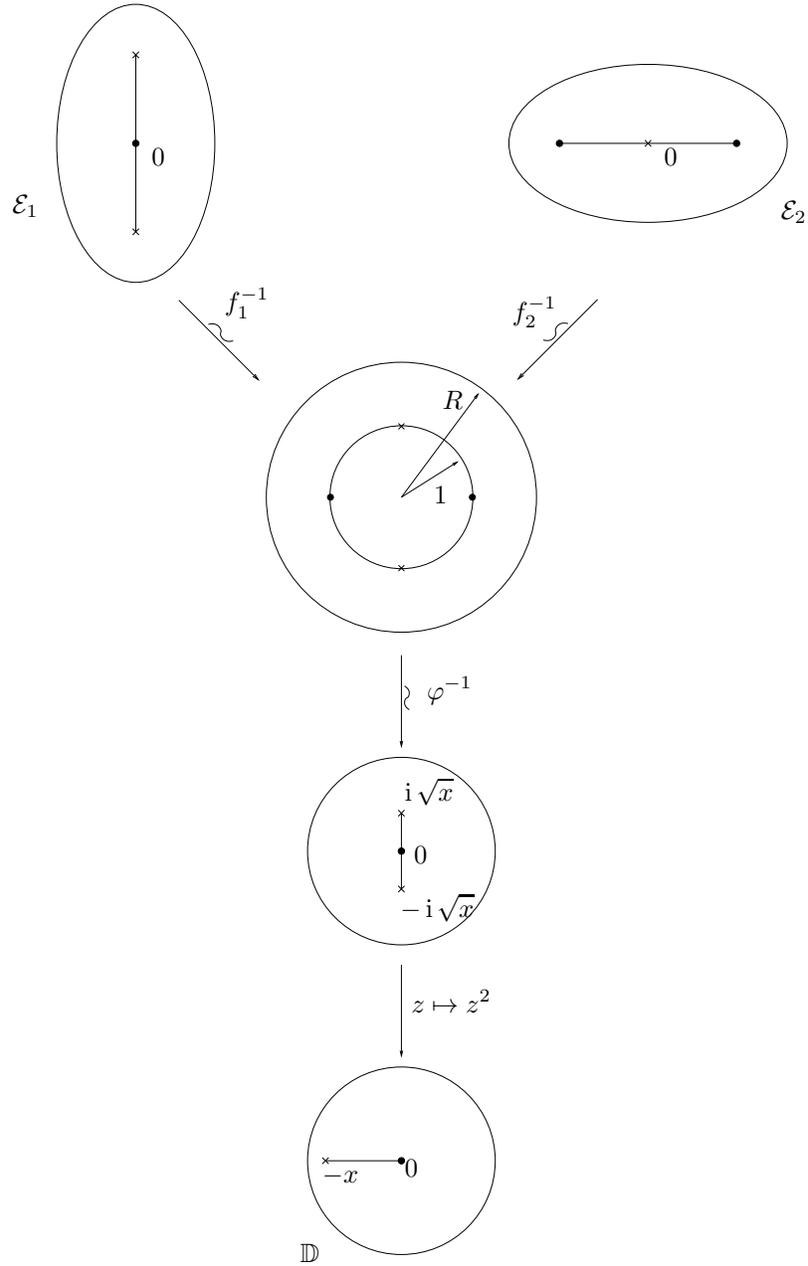}
\caption{We have two different covering spaces of $\mathbb{D}\setminus\left[-x, 0\right]$ given by $\mathcal{E}_1$ and $\mathcal{E}_2$. We distinguish by crosses and points the successive inverse images of $-x$ and $0$.}\label{fig4}
\end{figure}
\end{center}

We now have all the elements to solve our problem. We start by recalling the problem.  Let $f$ be a q.c. mapping from $\mathbb{D}$ to $\mathbb{D}$ such that $f\left( 0 \right)=-x$ and $f_{\mid_{\bd}}=\id_{\bd}$. Since $f$ maps the branch point $0$ to the branch point $-x$, we can lift it. We denote this lift by $\tilde{f}$ (see Figure \ref{final}). It is easy to check that $\tilde{f} : \mathcal{E}_1 \rightarrow \mathcal{E}_2$ is a q.c. mapping, with the same q.c. dilatation as $f$. Furthermore, $\tilde{f}\left( 0\right)=0$ and its restriction to the boundary coincides with the restriction to the boundary of the affine map $$\tilde{f}_0 : x+\ii y\mapsto \frac{R+\frac{1}{R}}{R-\frac{1}{R}}x+\ii\frac{R-\frac{1}{R}}{R+\frac{1}{R}}y .$$ 
 Since this map is symmetric with respect to $0$, we can conclude by construction of these two 2-sheeted ramified coverings that it descends to a q.c. mapping $f_{\left( 0, x \right)}$ whith the same q.c. dilatation as $\tilde{f}_0$. Furthermore, $f_{\left( 0, x \right)}$ satisfies the conditions of Problem \ref{mainpb} and  by using Proposition \ref{prop} we conclude that 
\begin{equation}\label{minoration}
\Kf_f \geq \left(\frac{R^2+1}{R^2-1}\right)^2,
\end{equation}
with equality if and only if $f=f_{\left( 0, x \right)}$. We call $f_{\left( 0, x \right)}$ the \textit{extremal map} for Problem \ref{mainpb}. 

It is important to note that the existence and the uniqueness of such a mapping cannot be deduced from the so-called \textit{Teichm\"uller theorem},  whose the first statement can be found in \cite{T20}.\footnote{ Teichm\"uller proved in \cite{T20} the uniqueness. The existence is stated there as a ``conjecture.''  Teichm\"uller proved existence (for closed surfaces) later  in \cite{T29}. We refer to the corresponding commentaries \cite{alberge&papadop&su} and \cite{T29C}.} Indeed, in the present case, all points (and not only a finite number) on the boundary are fixed.   

However, let us observe the following interesting fact. We recall  that the \emph{Beltrami differential}  associated with a q.c. mapping $f : \mathbb{D}\rightarrow\mathbb{D}$ is an element of $\textrm{L}^{\infty}\left( \mathbb{D} \right)$ which is defined by
$$
\mu_{f}:= \frac{\dzb\! f}{\dz\! f}.
$$
As $f_{\left( 0, x \right)}\circ p_1 = p_2 \circ \tilde{f}_0$, we conclude by using conformality of $p_i$ $\left( i=1,2 \right)$ that 
$$
\mu_{f_{\left( 0, x\right)}}\circ p_1 = \left( \frac{p_{1}^{\prime}}{\vert p_{1}^{\prime} \vert}\right)^2 \cdot k_{\left( 0,x\right)};
$$
where $k_{\left( 0, x\right)} = \frac{\Kf_{f_{\left( 0, x \right)}} +1}{\Kf_{f_{\left( 0, x \right)}}-1}$. In other words, we have
\begin{equation}\label{teichmullermapping}
\mu_{f_{\left( 0, x\right)}} = k_{\left( 0, x \right)}\cdot \frac{\overline{\phi}}{\vert \phi\vert},
\end{equation}
 where $\phi$ is a meromorphic function on $\mathbb{D}$ with a pole of order $1$ at $0$. Through Equality (\ref{teichmullermapping}), the knowlegeable reader will recognize the general expression of what we call the \textit{Teichm\"uller mapping}. By the way, there are works of Strebel where Equality (\ref{teichmullermapping}) is a consequence of the so-called \textit{Frame Mapping Criterion}. We refer to \cite{strebel, strebel2}. See also \cite{reichsurvey} (p. 124).  We have also to mention §159 of \cite{T29} where Teichm\"uller had already guessed that  the extremal map statisfies Equation (\ref{teichmullermapping}). In fact, Teichm\"uller explained that for a given homeomorphism of the disc (i.e. a condition for all boundary points) we can always extend this map to a map with the smallest q.c. dilation and which is related to a quadratic differential by an equation of type (\ref{teichmullermapping}). Note that this is at the idea of what is called the \textit{non-reduced} Teichm\"uller theory and for which Problem \ref{mainpb} is an example.

\begin{center}
\begin{figure}[ht]
\centering
\psfrag{0}{$0$}
\psfrag{-x}{$-x$}
\psfrag{D}{$\mathbb{D}$}
\psfrag{f}{$f$}
\psfrag{f1}{$\tilde{f}$}
\psfrag{p}{$p_1$}
\psfrag{q}{$p_2$}
\psfrag{E1}{$\mathcal{E}_1$}
\psfrag{E2}{$\mathcal{E}_2$}
\includegraphics[width=0.7\linewidth]{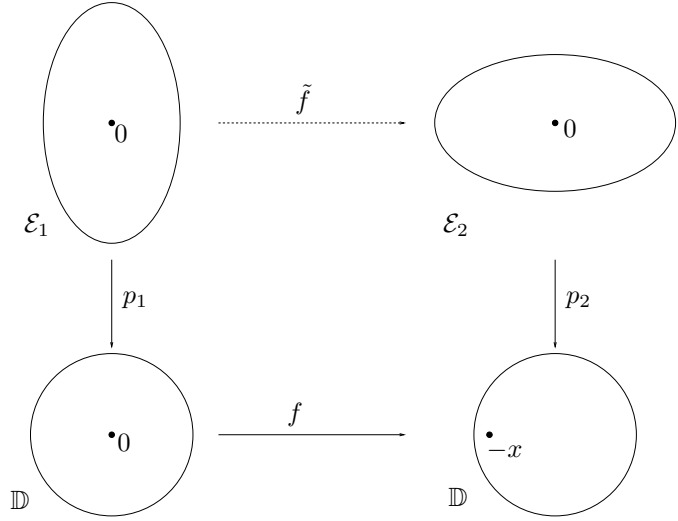}
\caption{A q.c. mapping $f$ from $\mathbb{D}$ to $\mathbb{D}$ such that $f\left( 0 \right)=-x$ can be lifted to a q.c. mapping with the same q.c. dilatation.}\label{final}
\end{figure}
\end{center}

\sloppy
Before getting further, let us note that to solve Problem \ref{mainpb}, we are in a situation  equivalent  to the Gr\"otzsch problem. Moreover, to solve the Gr\"otzsch problem for two rectangles $\mathcal{Q}_{1}\left( a_1 , b_1 , c_1 , d_1 \right)$ and $\mathcal{Q}_2 \left( a_2 , b_2 , c_2 , d_2 \right)$, we only need to consider q.c. mappings which send  $a_1$ to $a_2$, $b_1$ to $b_2$, $c_1$ to $c_2$ and $d_1$ to $d_2$. With this in mind, we see that if $f$ is a q.c. mapping from $\mathbb{D}$ to itself sending $0$ to $-x$, preserving the boundary and fixing $1$ and $-1$, then its lift $\tilde{f} : \mathcal{E}_1 \rightarrow \mathcal{E}_2$ sends the four extremal points of $\mathcal{E}_1$ to the four extremal points of $\mathcal{E}_2$. Thus, let us ask the following question:
\begin{prob} 
Is it possible to find and describe the quasiconformal mapping from $\mathbb{D}$ to itself such that 
\begin{itemize}
\item the images of $1$, $-1$ and $0$ are respectively $1$, $-1$ and $-x$ where $0<x<1$,
\item its quasiconformal dilatation is smallest possible?
\end{itemize}
\end{prob}
\fussy

\subsection{First consequences}

We have just seen that the extremal map $f_{\left(0, x \right)} : \mathbb{D}\rightarrow\mathbb{D}$ for Problem \ref{mainpb} has a q.c. dilation $\Kf_{f_{\left(0, x \right)}}$ equal to $\left(\frac{R^2+1}{R^2-1}\right)^2$.  Moreover, as we wrote above, $R=\sqrt{\Phi\left( \frac{1}{x}\right)}$, and so
\begin{equation}\label{eq}
\Kf_{f_{\left(0, x \right)}}=\left( \frac{\Phi\left( \frac{1}{x}\right)+1}{\Phi\left( \frac{1}{x}\right)-1}\right)^2.
\end{equation}\label{dilatationexacte}
By using (\ref{asympt}) in (\ref{dilatationexacte}), we conclude that
\begin{equation}\label{equiv}
\Kf_{f_{\left(0, x \right)}}=_{0^+}1+x+\emph{o}_{0^+}\!\left( x \right),
\end{equation}
which means that 
$$
\frac{\Kf_{f_{\left(0, x \right)}} - 1-x}{x} \underset{x\rightarrow 0^+}{\longrightarrow} 0.
$$

On the other hand, the right hand side of  (\ref{encadr}) gives us
$$
\Kf_{f_{\left(0, x \right)}}>\left( 1+\frac{x}{2}\right)^2.
$$
The last inequality leads us to
\begin{corollaire}\label{cor}
Let $f : \mathbb{D}\rightarrow\mathbb{D}$ be a q.c. mapping such that $f_{\mid_{\partial\!\mathbb{D}}}=\id_{\partial\!\mathbb{D}}$ and $f\left( 0 \right)\in \left] -1, 0\right]$. Then
$$
\vert f\left(0\right)\vert\leq 2\left( \Kf_{f}^\frac{1}{2}-1\right).
$$
\end{corollaire}

 Note that there is a small mistake in Teichm\"uller's paper  which is considered as a ``misprint'' by Earle and Lakic in \cite{earle&lakic}. Indeed, Teichm\"uller forgot to take the power $2$ in the right hand side of Relation (\ref{eq}) and so he obtained a different asymptotic behaviour in (\ref{equiv}) and a different upper bound in Corollary \ref{cor}. The same 
mistake appears  in \cite{anecdotegehring, kra}.\footnote{ I. Kra informed the author of the way he discovered this mistake. Kra needed some consequences of Teichm\"uller results to write \cite{kra} and he used for this the Gehring paper \cite{anecdotegehring}. Gehring found this error only after publishing \cite{anecdotegehring} and when he knew that Kra used his paper, he informed him.} About this mistake, we refer also to the editor footnote of \cite{T29}.

We denote the hyperbolic distance\footnote{We use the metric with constant curvature $-1$.} on the disc by $d_{\mathbb{D}}\left(. , . \right)$. According to (\ref{dilatationexacte}),  we can express the q.c. dilatation of $f_{\left(0, x \right)}$ with respect to this distance by the following formula:

\begin{equation}\label{eqdist}
\log \left(  \Kf_{f_{\left(0, x \right)}} \right) = 2\cdot d_{\mathbb{D}}\left( 0, \frac{1}{\Phi\left( \frac{1}{x} \right)}\right).
\end{equation}

\section{Some applications}

In this section, we mention some applications of  Teichm\"uller's result obtained by various authors.

\subsection{Kra's distance} Before setting the Kra distance, we show an easy extension of  Teichm\"uller's result. By extension, we mean to find for any distinct pair of points in $\mathbb{D}$ the q.c. mapping from $\mathbb{D}$ to $\mathbb{D}$ sending one point to the other, keeping the boundary pointwise fixed and with the smallest q.c. dilatation.  

Let $z_1$ and $z_2$ be two distinct points in $\mathbb{D}$. We denote by $\varphi_{\left( z_1 , z_2 \right)}$ the biholomorphism of the disc which sends $z_1$ to $0$ and $z_2$ to  some point $-x$, $0<x<1$. It is easy to check that the extremal map\footnote{We mean here the map with the smallest q.c. dilatation, sending $z_1$ to $z_2$ and keeping the boundary pointwise fixed.} is 
$$
f_{\left(z_1 , z_2 \right)}=\varphi_{\left( z_1 , z_2 \right)}\circ f_{\left( 0, x \right)}\circ \varphi_{\left( z_1 , z_2 \right)}^{-1},
$$
where $f_{\left( 0,x \right)}$ is the previous extremal map. The Beltrami differential of $f_{\left( z_1 , z_2 \right)}$ is related to a meromorphic function on the disc with a pole of order $1$ at $z=z_1$ by a relation analogous to (\ref{teichmullermapping}). Strebel calls in \cite{strebel} such a mapping the \textit{Teichm\"uller shift}.

Furthermore, 
\begin{equation}\label{eqdistance}
d : \left( z_1 , z_2 \right) \in \mathbb{D}\times \mathbb{D} \mapsto \frac{1}{2}\log\left( \Kf_{f_{\left( z_1 , z_2 \right)}} \right)
\end{equation} 
defines a new distance on the disc. Moreover, as Kra observed in \cite{kra}, $d$ is a complete metric.

We have all the ingredients to define the Kra distance. Let $S$ be a hyperbolic Riemann surface of finite type $\left( g, n \right)$, where $g$ is the genus and $n$ the number of punctures. We recall that a hyperbolic Riemann surface is a Riemann surface whose  universal cover is the unit disc. This implies in particular that $S$ carries a hyperbolic metric. Kra defined in \cite{kra} a new distance on $S$ as follows. For any two points $x$ and $y$ in $S$, he sets 
\begin{equation}\label{def4}
d_{\textrm{Kr}}\left( x, y \right) := \frac{1}{2}\log{\inf_{f}{\Kf_f}},
\end{equation} 
where the infimum is taken  over all q.c. mappings\footnote{A q.c. mapping on a Riemann surface is a mapping whose a lift to the universal cover is a q.c. mapping and the q.c. dilatation is the q.c. dilatation of this lift.} $f$ isotopic to the identity mapping and sending $x$ to $y$. This distance is now called the \textit{Kra distance}.\footnote{It seems that this name  appears for the first time in \cite{shen}.} From a compactness property of q.c. mappings we know that there always exists a q.c. mapping which attains the infimum in (\ref{def4}). Kra obtained the uniqueness of such a mapping if $x$ and $y$ are close enough for the hyperbolic metric on $S$ (see Proposition 6. in \cite{kra}). Furthermore, he showed that $d_{\textrm{Kr}}$ is equivalent to the hyperbolic metric but not proportional to it unless $S$ is the thrice-punctured sphere. In this exceptional case the two metrics coincide. It is of interest to note that the idea of Kra's distance already exists in \cite{T20}. Indeed, Teichm\"uller introduced such a distance and he showed in §27  that it coincides with the hyperbolic distance in the case of the thrice-punctured sphere. In the same paper, Teichm\"uller  explained in §160 that up to a condition, $S$ equipped with $d_{\textrm{Kr}}$ is a \textit{Finsler manifold}.

\subsection{About a problem of Gehring and a little bit more}

The \emph{Gehring problem}, which could be seen as a dual of Problem \ref{mainpb}, is the following. Given $K>1$, we want to describe the  value

\begin{equation}\label{eqgehring}
h_{\mathbb{D}}\left( K \right):= \sup{\left\lbrace d_{\mathbb{D}}\left( z, f\left( z \right) \right) \, \mid \, z\in\mathbb{D} \textrm{ and } f\in\mathcal{Q}_{\mathbb{D}}\left( K \right) \right\rbrace};
\end{equation}
where $\mathcal{Q}_{\mathbb{D}}\left( K \right)$ denotes the set of all $K$-q.c. mappings from $\mathbb{D}$ to $\mathbb{D}$ which hold the boundary  pointwise fixed.

This problem can be addressed for any planar domain $\Omega$ with at least $3$ boundary points. Indeed, for such a domain, we know that the  universal cover is the unit disc, so by pushing forward the hyperbolic metric on the disc, we can define a hyperbolic metric of constant curvature $-1$ on $\Omega$. We denote it by $d_{\Omega}\left( \cdot , \cdot \right)$ and we may want to determine the value of (\ref{eqgehring}) by considering $d_{\Omega}$ instead of $d_{\mathbb{D}}$.

Krzyz gave a value for (\ref{eqgehring}) in \cite{krzyz}. He proved that there exists $z_0 \in\mathbb{D}$ and $f_K \in\mathcal{Q}_{\mathbb{D}}\left( K \right) $ such that 
$$
h_{\mathbb{D}}\left( K \right)=d_{\mathbb{D}}\left( z_0, f_{K}\left( z_0 \right) \right).  
$$
He gave a precise value of $h_{\mathbb{D}}\left( K \right)$ and he showed, by using an analogue of Corollary \ref{cor}, that $f_K$ is the extremal map with respect to the Teichm\"uller problem (i.e. the extension of Problem \ref{mainpb} where the pair of points is $\left( z_0 , f_{K}\left( z_0 \right) \right)$.

Later, Solynin and Vuorinen showed in \cite{vuorinen4} that the supremum of (\ref{eqgehring}) is attained for a unique map, the map  described above.

The Gehring problem can also be addressed  for domains in $\mathbb{R}^n$, where $n>2$. See for example \cite{vuorinen1} and \cite{vuorinen2}. These two papers are related to a paper of Martin \cite{martinams}. Furthermore, Martin worked in \cite{martinmeandistorsion} on an extremal problem  close to Teichm\"uller's one. To be more precise, he considered for $0\leq x <1$, the value

\begin{equation}\label{inf}
\inf{\left\lbrace \frac{1}{\pi} \iint_{\mathbb{D}}{\Kf_{f}\left( z \right)\frac{1}{2} dz\wedge d\bar{z}} \, \mid \, f \textrm{ is q.c., } f\left( 0\right)=-x \textrm{ and } f_{\mid_{\bd}}=\id_{\bd} \right\rbrace}.
\end{equation}

He showed that if $x>0$, the infimum in (\ref{inf}) cannot be attained by a q.c. mapping.


\end{document}